\documentclass[10pt,A4paper,fleqn]{amsart}
\usepackage{mathrsfs}
\usepackage{amsfonts}
\usepackage{txfonts}
\usepackage{amsmath}
\usepackage{amssymb}
\usepackage{amsthm}
\usepackage[pdftex]{graphicx}
\usepackage[toc,page,title,titletoc,header]{appendix}
\usepackage{geometry}
\usepackage{upgreek}
\usepackage[T1]{fontenc}
\usepackage{color}
\usepackage{hyperref}
\usepackage[all]{xy}
\usepackage[french,english]{babel}
\usepackage{color}
\usepackage{cancel}
\usepackage{ulem}

\theoremstyle{plain}
\newtheorem{theo}{Theorem}[section]
\newtheorem*{theo*}{Theorem}
\newtheorem{prop}{Proposition}[section]

\newtheorem*{lem*}{Lemma}

\theoremstyle{definition}
\newtheorem{lem}{Lemma}[section]
\newtheorem{defi}{Definition}[section]

\theoremstyle{remark}

\newtheorem*{rem*}{Remark}
\newtheorem{rem}{Remark}[section]

\newcommand{\R}{\mathbb{R}}
\newcommand{\T}{\mathbb{T}}
\newcommand{\Z}{\mathbb{Z}}

\begin{document}

\title{A common first integral from three-body secular theory and Kepler billiards}

\author{Gabriella Pinzari$^{1}$,
Lei Zhao$^{2}$}
\address{$^{1}$ Department of Mathematics, University of Padova, Via Trieste, 63 - 35121 Padova, Italy}
\email{pinzari@math.unipd.it}

\address{$^{2}$School of Mathematical Sciences, Dalian University of Technology, Dalian 116000, China }
\email{zhao1899@dlut.edu.cn; lei.zhao@math.uni-augsburg.de}

%\thanks{Supported by }

\date\today

\abstract 
A particular first integral of the partially-averaged system in the secular theory of the three-body problem appears also as an important conserved quantity of integrable Kepler billiards. In this note we illustrate their common roots with the projective dynamics of the two-center problem. We then combine these two aspects to define a class of integrable billiard systems on surfaces of constant curvature.

\endabstract

\maketitle

%\tableofcontents

%\section{A general setting}

%Let $(M, \omega, H)$ be a Hamiltonian system. This means that $M$ is an even-dimensional smooth manifold equipped with a non-degenerate 2-form $\omega$, and a smooth function $H: M \to \R$. The Hamiltonian vector field $X_{F}$ of a function $F$ is defined by the equation 
%$$\omega(X_{F}, \cdot)=-d F.$$
%The Poisson bracket is defined as 
%$$\{F, G\}:=\omega(X_{F}, X_{G}).$$

%\begin{prop}\label{prop: 1.1} Suppose on a symplectic manifold $(M, \omega)$ and two functions $H_{1}, H_{2}$ such that there is a set of canonical coordinates $(\Lambda, \lambda, p_{i}, q_{i})$ defined on an open subset $\tilde{M} \subset M$, such that $H_{1}=H_{1} (\Lambda)$ and $H_{2}=H_{2} (\Lambda, p_{i}, q_{i})$. Then in $\tilde{M}$ there holds
% $$\{H_{1}, H_{2}\}=0.$$
%\end{prop}

%\begin{proof} Compute this with respect to the above-mentioned chart. We have by assumptions 
%$$\dfrac{\partial H_{1}}{ \partial \lambda}=\dfrac{\partial H_{2}}{ \partial \lambda}=\dfrac{\partial H_{2}}{\partial p_{i}}=\dfrac{\partial H_{2}}{\partial q_{i}}=0.$$
%Thus
%$$\{H_{1}, H_{2}\}=\dfrac{\partial H_{1}}{ \partial \Lambda} \dfrac{\partial H_{2}}{ \partial \lambda}-\dfrac{\partial H_{1}}{ \partial \lambda} \dfrac{\partial H_{2}}{ \partial \Lambda}+\cdots=0.$$
%\end{proof}

%Note that we may either have $\lambda \in \R$ or $\Lambda \in S^{1}$. Also there is no need to further specify the coordinates $(p_{i}, q_{i})$ besides that they complement the pair $(\Lambda, \lambda)$ to a set of canonical coordinates.

\section{Introduction}

In a Euclidean space of dimension at least $2$, the Kepler problem is super-integrable: It possesses more independent first integrals for the system to be integrable in the sense of Liouville. Concretely, in addition to the total energy and the angular momentum, additional conserved quantities are provided by the components of the Laplace-Runge-Lenz vector. Dynamically, this implies that the system is properly-degenerate: All bounded, non-singular orbits are closed. Out of $d$-frequencies (say, in a proper set of Action-Angle coordinates), only $1$-frequency is non-zero. This proper-degeneracy of the Kepler problem is the main reason for many degeneracies in celestial mechanics. The perturbation theory to Kepler problems is therefore hard to investigate. Nevertheless, this also leaves the possibility of considering systems with varying Kepler orbits, in a discrete or a continuous way, without destroying the integrability, even with fixed Keplerian energy. 

The function 
\begin{equation}\label{eq: D}
D=C^{2}-2 h A_{1}
\end{equation}
serves as a conserved quantity in addition to the energy for two types of these systems recently studied, in which $C$ is the total angular momentum, $A_{1}$ is a component of the Laplace-Runge-Lenz vector, and $h$ is a parameter: 

\begin{itemize}
\item Keplerian orbits evolving under a partially-averaged Newtonian potential, in $\R^{2}$ and $\R^{3}$. This has been noticed by the first author in \cite{Pinzari1}. In this situation, up to certain normalization, $h$ denotes half of the distance between the origin and the position of the second, motionless particle. $A_{1}$ is taken as the component of the Laplace-Runge-Lenz vector in the direction of the second particle.

\item In a planar Kepler billiard system, with a line as a wall of reflection. This has been first proved by Gallavotti-Jauslin in \cite{GJ}. In this case, $h$ denotes the distance from the Kepler center to the wall, and $A_{1}$ is the vertical component of the Laplace-Runge-Lenz vector with respect to the wall of reflection.
\end{itemize}

The study in \cite{Pinzari1} is continued in \cite{Pinzari2, Pinzari3, RuzzaPinzari1} toward further applications to the study of three-body problem. An apparently different first integral for close encounters in the regularized circular restricted three-body problem has been discussed in \cite{leviC1916, cardinG2022}, respectively for the planar and spatial case.

The existence of $D$ in the study of Gallavotti-Jauslin \cite{GJ} is used to disprove an ergodic assertion in a paper of Boltzmann \cite{Boltzmann}. This has intrigued many studies ever since \cite{Felder, GR, Zhao1}. The function $D$ is shown to hold as well for Kepler billiard systems with a branch of conic section as reflection wall, provided that the Kepler center lies at a focus of the conic section, $h$ is the focus-center distance of the conic section, and $A_{1}$ the component of the Laplace-Runge-Lenz vector along the axis of foci \cite{TakeuchiZhao1, TakeuchiZhao3}. 

In particular, in the system that Gallavotti-Jauslin considered, it is helpful to put an additional massless Kepler center symmetrically with respect to the line of reflection.  This line can be considered as a degeneracy along a family of confocal hyperbola with foci at the Kepler centers. Then $h$ is again the half distance between the Kepler centers. 

A main observation of this article is that the appearance of $D$ in both of systems can be understood in terms of projective dynamics \cite{Albouy}. {See also \cite{PZ} in which we present an alternative computation with full details.}
We make several further generalizations:

\begin{itemize}
\item In \cite{Pinzari1}, the partially-averaged Newtonian potential is defined with respect to a Kepler map. We propose a coordinate-free formulation. 
\item We generalize the theorem of \cite{Pinzari1} on the sphere, and in the hyperbolic space.
\item We generalize \cite{Pinzari1} to the partially-averaged sum of a Newtonian potential and a properly placed Hooke potential, for the system defined in a Euclidean space, on a hemisphere, and in the hyperbolic space.
\item We define new integrable mechanical billiard systems based on the partially-averaged system.
\end{itemize}

This note is organized accordingly.

\section{The Kepler and two-center problems}

\subsection{Kepler and two center problems on constant curvature spaces}

We consider the two-center problem in an Euclidean space $\R^{d}$. In this system, a moving particle of unit mass $q$ moves in the space $\R^{d}$ under the Newtonian attractions of two fixed Kepler centers $Z_{1}, Z_{2} \in \R^{d}$. If the masses of the centers are $m_{1}$ and $m_{2}$ respectively, then the equations of motion is
$$\ddot{q}=-\nabla_{q} V(q),$$
with the potential 
$$V(q)= -\dfrac{m_{1}}{\|Z_{1}-q\|}-\dfrac{m_{2}}{\|Z_{2}-q\|}.$$

This system is an integrable system, already known to Euler \cite{Euler} and Lagrange \cite{Lagrange} for $d=2, 3$. Indeed by Noether's theorem, rotations around the axis of the centers provide $d-2$ first integrals commuting with the energy in $\R^{d}$. The integrability of the system follows from one additional non-trivial first integral, which has been identified by these authors. 

This type of problems can be imposed on the sphere as well. A potential system on the unit sphere $\mathbb{S}^{d}=\{x: ||x||=1\} \subset \R^{d+1}$ with potential function $V$ describing the motion of a unit mass particle $q$ has equations of motion 
$$\nabla_{\dot{q}} \dot{q}=-\hbox{grad}_{g_{0}} V,$$
in which both the covariant derivative and the gradient is defined with respect to the round metric $g_{0}$ on $\mathbb{S}^{d}$. This is a natural mechanical system, and the total energy 
$$\dfrac{1}{2} g_{q}(\dot{q}, \dot{q})+V(q)$$
is a conserved quantity. 

As noted by Graves \cite{Graves} and Serret \cite{Serret}, the natural extension of Kepler problem from the Euclidean space to the (unit) sphere which preserves the elliptic law is given by a potential proportional to the cotangent of the central angle the moving particle made with the Kepler center. The spherical Kepler potential respect to a center $Z_{1} \in \mathbb{S}^{d}$ of mass $m_{1}$ is thus given by $-m_{1} \cot{\theta_{1}}$. The angle $\theta_{1}$ is the central angle the particle $q$ made from $Z_{1}$. Note that the antipodal point of $Z_{1}$ on $ \mathbb{S}^{d}$ becomes another singularity for the potential, repulsive if $Z_{1}$ is attractive and attractive if $Z_{1}$ is repulsive, so Kepler singularities on the sphere appear in pairs.  It is known to these authors that all non-singular orbits are spherical ellipses, which presents the form of elliptic law on the sphere. Further study shows that there are even more similarities of the Euclidean and spherical Kepler problems \cite{KH, AZ}. 

As in the Euclidean case, the spherical 2-center problem describes the motion of a particle moving on $\mathbb{S}^{d}$ attracted by two pairs of antipodal centers of Kepler type. On the sphere $S^{d}$ this system is thus described by a potential of the form $-m_{1} \cot{\theta_{1}}-m_{2} \cot{\theta_{2}}$. This is an integrable system closely related to the Euclidean two-center problem \cite{Albouy}, and will be important for this note.

We may define Kepler problem and the two-center problems in the hyperbolic spaces in an analogous manner. They share many similarities with their spherical and Euclidean analogues. We note that the two-center problem in the hyperbolic 3-space has been first studied by Killing \cite{Killing}.

\subsection{Projective Dynamics of the Two-Center/Lagrange Problems}
In \cite{Albouy}, Albouy gave an explanation that the additional non-trivial first integral of the two center problem in $\R^{d}$ is ultimately related to the energy of a \emph{corresponding system} defined on the sphere, which is the two-center problem on a hemisphere of $\mathbb{S}^{d}$. It is shown that the natural central projection from the center of $\mathbb{S}^{d}$ projects unparametrized orbits of one system to unparametrized orbits of the other and thus both energies of these two systems give rise to a pair of independent, commuting first integrals for both systems.

Here we recall the construction. We consider $\mathbb{S}^{d}$ embedded as unit sphere in $\R^{d+1}$, with $\R^{d}$ embedded as the hyperplane $\R^{d} \times \{-1\} \subset \R^{d+1}$. The lower hemisphere $\mathbb{S}^{d}_{SH}:=\{(x_{1}, \cdots, x_{d+1}) \in \mathbb{S}^{d}, x_{d+1}<0\}$ of $\mathbb{S}^{d}$ is thus projected to $\R^{d}$ via central projection from the origin in $\R^{d+1}$. If the centers of the two center problem are put symmetrically with respect to the contact point of $\mathbb{S}^{d}$ with $\R^{d} \times \{-1\}$, then according to  \cite{Albouy}, this projection also project unparametrized orbits of this spherical problem to unparametrized orbits of the two-center problem in $\R^{d}$, defined with respect to a modified metric $\|\cdot\|_{a}$ given below. The spherical energy thus induces an additional first integral for the two-center problem in $\R^{d}$ and vice versa, the energy of the Euclidean two-center problem is an additional first integral for the spherical two-center problem. Moreover, we may as well put a harmonic potential in the middle of the two Kepler centers and all these hold for the modified system, called the Lagrange problem \cite{Lagrange} as well. This can as well be defined on $\mathbb{S}^{d}_{SH}$, with a potential proportional to the squared tangent of the central angle as the spherical harmonic potential.

We denote the masses of the two Kepler centers and the string constant of the Hooke center in $\R^{d}$ as $m_{1}, m_{2}, f \in \R$ respectively. We do not make assumptions on signs of them. We put the centers in $\R^{d}$ at $(0, \cdots, 0, \pm a)$. Notice that in $\R^{d}$, the mechanical systems are defined with respect to the non-standard affine norm $\|(v_{1},\cdots, v_{d})\|_{a}:=\sqrt{v_{1}^{2}+\cdots + v_{d-1}^{2}+\frac{v_{d}^{2}}{1+a^{2}}}$.

We recall the following proposition from \cite{Albouy, TakeuchiZhao1}:

\begin{prop}\label{prop: 2.1} The Lagrange problem in $\R^{d}$ with the metric from the norm $\|\cdot\|_{d}$ and with parameters $(m_{1}, m_{2}, f)$ and on $\mathbb{S}^{d}_{SH}$ with the round metric inherited from $\R^{d+1}$ and with parameters $(m_{1} \sqrt{1+a^{2}}, m_{2} \sqrt{1+a^{2}}, f)$  are in projective correspondence. 
\end{prop}

These systems, as well as those defined on $\mathbb{S}^{d}_{SH}$ are thus integrable. Similar result holds between systems defined in $\mathbb{S}^{d}$ and in $\mathbb{H}^{d}$ as well. We refer to \cite{TakeuchiZhao1} for precise statements. Note that setting $f=0$ in Prop \ref{prop: 2.1} we obtain the relevant result in the two-center problem.

For the problem in $\R^{d}$, we denote the energy of the problem as $E_{Eucl}$. We consider $\R^{d}$ as a chart for the lower hemi-sphere $\mathbb{S}_{SH}$ and we denote by $E_{sph}$ the spherical energy in terms of positions and velocities, so as a function on the tangent bundle $T \R^{d}$. Note that Legendre transformation allows to identify the tangent and cotangent bundles, which then equips the tangent bundle with an induced symplectic form and an Poisson bracket.

\subsection{Relationship among first integrals for the two center problem}
{We consider the Euclidean two center problem and argue in the case $d=2$. 
Due to the rotation invariance the formulas hold for any $d \ge 2$ as well. }
{We follow the argument of \cite{Albouy} in this section. As mentioned in the introduction, we refer to \cite{PZ} for an alternative report with full computational details.}

%The argument is provided for the planar case $d=2$, which easily carried to more general dimensions. 

We put the Kepler centers at $(0, a)$ and $(0, -a)$ in $\R^{2}$, with masses $m_{1}$ and $m_{2}$ respectively. We equip $\R^{2}$ with the norm $\|\cdot\|_{a}$, defined as  $\|(u, v)\|_{a}=\sqrt{u^{2}+\frac{v^{2}}{1+a^{2}}}$.

This is the same setting as in \cite{Zhao1}, in which only one effective Kepler center is considered. By Prop \ref{prop: 2.1} we may as well add the second Kepler center center to obtain the two-center problem. The energy of the Euclidean two-center problem is
\begin{equation}\label{eq: Euclidean energy 2CP}
E_{Eucl}=\dfrac{1}{2} \Bigl(\dot{x}^{2}+\dfrac{\dot{y}^{2}}{1+a^{2}}\Bigr)-\dfrac{m_{1}}{\sqrt{x^{2}+\frac{(y-a)^{2}}{1+a^{2}}}}-\dfrac{m_{2}}{\sqrt{x^{2}+\frac{(y+a)^{2}}{1+a^{2}}}},
\end{equation}
the energy of the corresponding (hemi-)spherical 2-center problem can be written, with the same kind of computation as in \cite{Zhao1}, as
\begin{equation}\label{eq: spherical energy}
E_{sph}=\dfrac{1}{2} \Bigl(\dot{x}^{2}+\dot{y}^{2})+\dfrac{1}{2} (x \dot{y}-\dot{x} y)^{2}-\dfrac{m_{1}\sqrt{1+a^{2}} (a y+1)}{\sqrt{(y-a)^{2}+(1+a^{2}) x^{2}}} -\dfrac{m_{2}\sqrt{1+a^{2}} (-a y+1)}{\sqrt{(y+a)^{2}+(1+a^{2}) x^{2}}}.
\end{equation}
We use the affine coordinate change $(\xi, \eta) \mapsto (x=\xi, y=\sqrt{1+a^{2}} \eta + a)$ to bring the norm on $\R^{2}$ into the standard Euclidean one.
{For a quick comparison with \cite{PZ} we remark that that such change transforms the unit sphere into an ellipsoid. 
}
 With this change we have put the second center at $(0, -2 a/\sqrt{1+a^{2}})$ in this new coordinates. And we have

$$E_{Eucl}=\dfrac{1}{2} \Bigl(\dot{\xi}^{2}+\dot{\eta}^{2}\Bigr)-\dfrac{m_{1}}{\sqrt{\xi^{2}+\eta^{2}}}-\dfrac{m_{2}}{\sqrt{\xi^{2}+(\eta+2 a/\sqrt{1+a^{2}})^{2}}},$$

$$
E_{sph}=(1+a^{2}) \Bigl(\dfrac{1}{2}  (\dot{\xi}^{2}+\dot{\eta}^{2}) -\dfrac{m_{1}}{\sqrt{\xi^{2}+\eta^{2}}} \Bigr) + \dfrac{(1+a^{2})}{2} \Bigl((\xi \dot{\eta} - \eta \dot{\xi})^{2}-\dfrac{2 a}{\sqrt{1+a^{2}}}((\xi \dot{\eta} - \eta \dot{\xi}) \dot{\xi}+\dfrac{m_{1} \eta}{\sqrt{\xi^{2}+\eta^{2}}})\Bigr) -\dfrac{m_{2} (-a \sqrt{1+a^{2}} \eta+1-a^{2})}{\sqrt{(\eta + 2a/\sqrt{1+a^{2}})^{2}+\xi^{2}}}.
$$

Further set $E_{Kep}:=E_{Eucl}|_{m_{2}=0}$ the Keplerian energy with respect to the first center.  Then we can write 
$$E_{Eucl}=E_{Kep}+V_{Eucl, 2}:=E_{Kep} -\dfrac{m_{2}}{\sqrt{\xi^{2}+(\eta+2 a/\sqrt{1+a^{2}})^{2}}}.$$

From this it is clear that

\begin{equation}\label{eq: planar spherical energies}
E_{sph}=(1+a^{2})(E_{Kep}+V_{Eucl, 2}+(D+K)/2).
\end{equation}

in which 
$$D=(\xi \dot{\eta} - \eta \dot{\xi})^{2}-\dfrac{2 a}{\sqrt{1+a^{2}}}\Bigl((\xi \dot{\eta} - \eta \dot{\xi}) \dot{\xi}+\dfrac{m_{1} \eta}{\sqrt{\xi^{2}+\eta^{2}}}\Bigr)$$
as appeared in \eqref{eq: D} is a first integral of $E_{Kep}$: This is the first integral whose role we wish to address in this note. The term $V_{Eucl, 2}$ is linear in $m_{2}$, as well as $K$, that we observe from its explicit expression
$$K:= -\dfrac{2}{1+a^{2}}\dfrac{m_{2} (a \sqrt{1+a^{2}} \eta+2 a^{2})}{\sqrt{(\eta + 2a/\sqrt{1+a^{2}})^{2}+\xi^{2}}}.$$

%=\dfrac{2}{(1+a^{2})} ((1+a^{2}) V_{Eucl, 2}-V_{sph, 2})=2 V_{Eucl, 2}-\dfrac{2}{1+a^{2}} V_{sph, 2}.

%$$(1+a^{2})V_{Eucl, 2}:=-\dfrac{(1+a^{2} )m_{2}}{\sqrt{x^{2}+\frac{(y+a)^{2}}{1+a^{2}}}}=-\dfrac{(1+a^{2})  m_{2} }{\sqrt{\xi^{2}+(\eta+2 a/\sqrt{1+a^{2}})^{2}}}$$

%$$V_{sph, 2}:=-\dfrac{m_{2}\sqrt{1+a^{2}} (-a \sqrt{1+a^{2}} \eta+1-a^{2})}{\sqrt{(\sqrt{1+a^{2}} \eta + 2a)^{2}+(1+a^{2}) \xi^{2}}}=-\dfrac{m_{2} (-a \sqrt{1+a^{2}} \eta+1-a^{2})}{\sqrt{(\eta + 2a/\sqrt{1+a^{2}})^{2}+\xi^{2}}}$$

%Difference, with factor $\dfrac{2}{1+a^{2}}$,

%$$K:= -\dfrac{2}{1+a^{2}}\dfrac{m_{2} (a \sqrt{1+a^{2}} \eta+2 a^{2})}{\sqrt{(\eta + 2a/\sqrt{1+a^{2}})^{2}+\xi^{2}}}=\dfrac{2}{(1+a^{2})} ((1+a^{2}) V_{Eucl, 2}-V_{sph, 2})=2 V_{Eucl, 2}-\dfrac{2}{1+a^{2}} V_{sph, 2}.$$

%$$E_{sph}=(1+a^{2}) (E_{Eucl} +(D+K)/2).$$

The spherical energy $E_{sph}$ is conserved along orbits of the Euclidean two-center problem. Remind that we have used the Legendre transformation to identify the tangent bundle and the cotangent bundle. We may thus use the Legendre transformation to transport the Poisson bracket from the cotangent bundle to the tangent bundle for our convenience. This bracket is denoted still by $\{\cdot, \cdot\}$. We thus have

\begin{equation}\label{eq: commutativity of energies}
\{E_{Eucl}, E_{sph}\}=0.
\end{equation}

Thus we get

\begin{equation}\label{eq: Poisson 1}
\{E_{Kep}+V_{Eucl, 2}, D+K\}=0.
\end{equation}

\subsection{Averaging in the two center problem}

Remind that we have identified the tangent and cotangent bundles by means of Legendre transformation. We continue to work on the tangent bundle, which is equipped with a symplectic form and a Poisson-bracket induced from the cotangent bundle.

We set $m_{1}>0$ to ensure the existence of non-singular negative energy orbits of $E_{Kep}$. These orbits are all closed.  We argue on the corresponding region of the tangent bundle, on which $E_{Kep}$ defines a Hamiltonian circle action and that these orbits in $\R^{d}$ avoid the second Kepler center. Denote the set of orbits of $E_{Kep}$ in $\T \R^{d}$ whose projection passes through the second Kepler center by $\mathcal{I}$. We shall exclude this set as $V_{Eucl, 2}$ encounters singularities along these orbits. Then the region $\mathcal{R}$ is defined by having negative Keplerian energy $E_{Kep}<0$ and non-zero Kepler angular momentum $C_{Kep} \neq 0$ to exclude collisional orbits in $T \R^{d} \setminus \mathcal{I}$:

$$\mathcal{R}:=\{E_{Kep}<0, C_{Kep} \neq 0\} \subset T \R^{d} \setminus \mathcal{I}.$$

\begin{rem} The latter condition excluding collision orbits can likely be dropped since averaging is well-defined up to collisional orbits by mean of regularization. In dimension 2 and 3 this can be justified using methods and coordinates explained in \cite{Fejoz} and \cite{ZhaoKS} respectively. This is likely the case for all dimensions but we shall not deal with this issue here. 
\end{rem}

%The latter requirement can be dropped by using a proper regularization procedure. 

\begin{defi}\label{defi: Euclidean average} The partially-averaged system is the system defined on $\mathcal{R}$ by the Hamiltonian 
$$E_{Eucl, a}=E_{Kep}+\int_{S^{1}} V_{Eucl, 2} \, d s, $$
in which the integration is understood as averaging over closed orbits of $E_{Kep}$, parametrized by the circle $S^{1}=\R/\Z$.
\end{defi}

The definition does not require the choice of a set of canonical coordinates. When $d=2, 3$, several sets of canonical coordinates are available, including the classical Delaunay variables and Poincar\'e coordinates. Moser constructed action-angle coordinates of Delaunay type for $\R^{d}$ \cite{MoserZehnder}. In general it is not always easy to construct these coordinates, which is even more complicated in the case when more bodies were involved \cite{PinzariCoordinates}.  Moreover, all these coordinates become singular at certain types of orbits. 

Nevertheless, in a local trivialization of this circle bundle, we can have an (abstract) set of canonical coordinates $(L, \tilde{\ell}, u_{1}, \cdots, u_{d-1}, v_{1}, \cdots v_{d-1})$ after a further localization on the space of closed Keplerian orbits, in which $L$ is the circular angular momentum as in the Delaunay coordinates, and $\tilde{\ell}$ is an angle parametrizing the Keplerian orbits in this trivialization (shifted from the mean anomaly by a phase, when both are well-defined), and $(u_{i}, v_{i})$ are secular variables, in the sense that they are functions of the orbits only, independent of the actual position and velocity of the particle on the orbit. In \cite{Zhao2} one finds detailed discussion on these type of abstract local coordinates.

\begin{lem}\label{lem: 2.1} Near any closed Keplerian orbit of $E_{Kep}$ in $\mathcal{R}$ there exists a set of canonical coordinates $(L, \tilde{\ell}, u_{1}, \cdots, u_{d-1}, v_{1}, \cdots v_{d-1})$ such that $\tilde{\ell}$ is an angle parametrizing the closed Keplerian orbits, $L$ is its conjugate action variable on which $E_{Kep}$ only depends, and $(u_{1}, \cdots, u_{d-1}, v_{1}, \cdots v_{d-1})$ are secular coordinates which depends only on the Keplerian orbits. 
\end{lem}

\begin{proof} In a local trivialization we have the angle $\tilde{\ell}$ and its conjugate action $L$ (as a moment map of the Hamiltonian circle action). If we symplectically reduce a local trivialization of this circle bundle with respect to the Hamiltonian circle action, then by Darboux theorem there is a set of symplectic coordinates $(\tilde{u}_{1}, \cdots, \tilde{u}_{d-1}, \tilde{v}_{1}, \cdots \tilde{v}_{d-1})$. Then we recover that $(L, \tilde{\ell}, u_{1}, \cdots, u_{d-1}, v_{1}, \cdots v_{d-1})$ is a set of canonical coordinates, in which $u_{i}=f_{i} (L)\tilde{u}_{i}, v_{i}=g_{i} (L) \tilde{v}_{i}$ for some non-zero functions $f_{i}, g_{i}$ of $L$. Indeed these functions cannot depend on $\tilde{\ell}$ as a consequence of the unicity of the moment map of a Hamiltonian circle action up to a constant. It is possible to say more about the functions $f_{i}, g_{i}$ but we shall not need to do this for our current purpose. 
\end{proof}

It is not hard to see that Definition \ref{defi: Euclidean average} is independent of the choice of local coordinates. Thus averaging is well-defined in a non-trivial principal $S^{1}$-bundle, in our case defined by motions along closed Keplerian orbits. 

\begin{lem}\label{lem: 2.2} The expression $\{G_{1}(L), G_{2} (L, \tilde{\ell}, u_{1}, \cdots, u_{d-1}, v_{1}, \cdots v_{d-1}) \}$ has zero-average over a period of $\tilde{\ell}$.
\end{lem}
\begin{proof} In coordinates, 
$$\{G_{1}(L), G_{2} (L, \tilde{\ell}, u_{1}, \cdots, u_{d-1}, v_{1}, \cdots v_{d-1}) \}=G_{1}'(L)\, \dfrac{d}{d \tilde{\ell}} (G_{2}).$$
Integration of a total derivative over a period of $\tilde{\ell}$ is zero, and thus the conclusion.
\end{proof}

The following proposition is a coordinate-free formulation of a main observation in \cite{Pinzari1}. Without specifying the canonical coordinates used, the dimension of the space  and the sign of $m_{2}$ in the announcement we bring a small generalization.

\begin{prop}\label{prop: 1}
$D$ is a first integral of the partially-averaged system $E_{Eucl, a}$:
$$\{E_{Eucl, a}, D\}=0.$$
\end{prop}

\begin{proof}
We cover  $M$ by open submanifolds on which the above-mentioned canonical coordinates are defined. Let $\tilde{M}$ be any of these submanifolds, with local canonical coordinates $(L, \tilde{\ell}, u_{1}, \cdots, u_{d-1}, v_{1}, \cdots v_{d-1})$ as described in Lem \ref{lem: 2.1}. 

In any set of these coordinates, we may write 
$$E_{Eucl, a}=E_{Kep} +\langle V_{Eucl,2} \rangle_{\tilde{\ell}}:=E_{Kep}+\dfrac{1}{2 \pi}\int_{0}^{2 \pi} V_{Eucl,2} d \tilde{\ell}.$$

Then 
$$E_{Eucl}=E_{Eucl, a} + V_{Eucl,2}-\langle V_{Eucl,2}\rangle_{\tilde{\ell}}.$$

We apply averaging theory. There is only one frequency that we need to average. Moreover we may use $m_{2}$ to play the role of a small parameter. In this setting, we know that with one step of averaging we get a canonical transformation $\phi: \tilde{M} \mapsto \tilde{M}$ of order $O(m_{2})$ which conjugate $E_{Eucl}$ to $E_{Eucl, a}$ up to an error term of order $O(m_{2}^{2})$, that is to say, 
$$E_{Eucl, a} \dot{=} E_{Eucl} \circ \phi.$$
in which $\dot{=}$ denotes equal up to a term of order  $O(m_{2}^{2})$.
And
$$\{E_{Eucl, a}, D \circ \phi+K \circ \phi\} \dot{=}0.$$

Write 
$D \circ \phi+K \circ \phi=D + D \circ \phi-D+K \circ \phi$ and notice that 
$$ D \circ \phi-D+K \circ \phi=O(m_{2}),$$ since $\phi$ is $O(m_{2})$ close to identity. Also, $K$ is of order $m_{2}$. This means that 
$$D \circ \phi+K \circ \phi=D + O(m_{2}).$$
Thus
$$\{E_{Eucl, a}, D \circ \phi+K \circ \phi\}\dot{=}0=\{E_{Eucl, a}, D\}+O(m_{2}).$$
The expression $\{E_{Eucl, a}, D\}$ is independent of $\tilde{\ell}_{1}$, as well as its term of order $m_{2}$. The contribution of order $m_{2}$ in the remainder term $O(m_{2})$ comes from $\{E_{Kep}, D \circ \phi+K \circ \phi\}$, which has zero-average while integrated with respect to $\tilde{\ell}_{1}$ over $(0, 2 \pi)$ according to Lem \ref{lem: 2.2}. So by integration we obtain
$$\{E_{Eucl, a}, D\}\dot{=}0.$$
and thus 
$$\{E_{Eucl, a}, D \}=0,$$

Moreover, this equality is independent of the choice of $\tilde{M}$. Therefore it holds on all $M$. This ends the proof.
\end{proof}

Note that the proof does not rely on that the system is defined in $\R^{d}$ and can therefore be generalized to many other cases. 

\subsection{The Spherical Case}

We now consider the system on the sphere $\mathbb{S}^{d}$. The Kepler problem on the sphere has the nice and highly specific property that \emph{all} non-singular orbits are closed. Indeed these are all spherical ellipses. This allows us to average along all non-singular orbits of the spherical Kepler problem.

Again we restrict ourselves to the case $d=2$. To avoid complexification with coordinate changes we put the reference Kepler center at $(0, \dfrac{a}{\sqrt{1+a^{2}}}, -\dfrac{1}{\sqrt{1+a^{2}}})$, corresponding to the point $(0, a)$ in the chart.

We set 
$$E_{sKep}:=\dfrac{1}{2} \Bigl(\dot{x}^{2}+\dot{y}^{2})+\dfrac{1}{2} (x \dot{y}-\dot{x} y)^{2}-\dfrac{m_{1}\sqrt{1+a^{2}} (a y+1)}{\sqrt{(y-a)^{2}+(1+a^{2}) x^{2}}}$$
and write 
$$E_{sph}=E_{sKep}+V_{sph, 2}.$$
By \eqref{eq: spherical energy}, the term $V_{sph, 2}$ is linear in $m_{2}$. Note that by analyticity, all the terms and this relationship are well-defined on the whole sphere outside the four singular points of the individual spherical Kepler potentials, even though we have only deduced this formula in a chart for a hemisphere (see related discussions in \cite{TakeuchiZhao1}). This argument applies as well to what follows.

It follows from \eqref{eq: planar spherical energies} and \eqref{eq: commutativity of energies}, that the analogue of \eqref{eq: Poisson 1} holds in the spherical case: 

\begin{equation}\label{eq: Poisson 2}
\{E_{sKep}+V_{sph, 2}, D+K\}=0,
\end{equation}
in which the bracket is the Poisson bracket on $T \mathbb{S}^{2}$ transported from $T^{*} \mathbb{S}^{2}$ via Legendre transformation. Note that the same is true for $\mathbb{S}^{d}$ for any $d=1, 2, 3, \cdots$.

The subset in $T \mathbb{S}^{d}$ such that the spherical Kepler orbits are all non-singular is a dense open set in $T \mathbb{S}^{d}$. We denote by $\mathcal{I}_{Sph}$ by the set of orbits of $E_{sKep}$ in $T \mathbb{S}^{d}$ whose projections to $\mathbb{S}^{d}$ pass through other Kepler centers. We set $\mathcal{R}_{sph}=T \mathbb{S}^{d} \setminus \mathcal{I}_{Sph}$.

\begin{defi} The partially-averaged system is given on the averaging region $\mathcal{R}_{sph}$ by 
$$E_{sph, a}=E_{sKep}+\int_{S^{1}} V_{sph,2} \, d s, $$
in which the integration is made over closed orbits of $E_{sKep}$.
\end{defi}

Now with exactly the same proof to Prop. \ref{prop: 1}, with the Euclidean quantities replaced by their spherical counterparts, we get

\begin{prop}\label{prop: 2}
$D$ is a first integral of the partially-averaged system $E_{sph, a}$ on $\mathcal{R}_{sph}$:
$$\{E_{sph, a}, D\}=0.$$
\end{prop}

Note that on the sphere, the existence of abstract canonical coordinates can be shown in the same way as Lem \ref{lem: 2.1}, in which no Euclidean assumption has been made.

\section{Averaging in the Lagrange Problem}

We now consider the Lagrange problem in $\R^{2}$. The result and argument extends easily to other cases as well.  In $\R^{2}$,  two Kepler centers are put at $(0, a)$ and $(0, -a)$ in $\R^{2}$, with masses $m_{1}$ and $m_{2}$ respectively. In addition, a Hooke center is put at $(0,0)$, with a Hooke factor $f$. The metric on $\R^{2}$ is $\|\cdot\|_{a}$, defined as  $\|(u, v)\|_{a}=\sqrt{u^{2}+\frac{v^{2}}{1+a^{2}}}$.

The energy of this system is
\begin{equation}\label{eq: Euclidean energy LP}
\tilde{E}_{Eucl}=\dfrac{1}{2} \Bigl(\dot{x}^{2}+\dfrac{\dot{y}^{2}}{1+a^{2}}\Bigr)-\dfrac{m_{1}}{\sqrt{x^{2}+\frac{(y-a)^{2}}{1+a^{2}}}}-\dfrac{m_{2}}{\sqrt{x^{2}+\frac{(y+a)^{2}}{1+a^{2}}}}+f (x^{2}+\frac{y^{2}}{1+a^{2}}).
\end{equation}

As before we write

$$\tilde{E}_{Eucl}=E_{Kep} + \tilde{V}_{Eucl, 2},$$
with now 
$$\tilde{V}_{Eucl, 2}:=-\dfrac{m_{2}}{\sqrt{x^{2}+\frac{(y+a)^{2}}{1+a^{2}}}}+f \Bigl(x^{2}+\frac{y^{2}}{1+a^{2}}\Bigr).$$

According to Prop \ref{prop: 2.1} this system has a spherical corresponding system, which gives rise to an additional first integral deduced from the spherical energy, to which we need to add the term $f (x^{2}+y^{2})$ from the spherical harmonic potential. Comparing to \eqref{eq: planar spherical energies} only a term linear in $f$ has been added, so we are content to write

\begin{equation}
\tilde{E}_{sph}=(1+a^{2})(E_{Kep}+\tilde{V}_{Eucl, 2}+(D+\tilde{K})/2).
\end{equation}
in which now $\tilde{K}$ is a linear combination of terms linear in $m_{2}$ and $f$ respectively.

We assume as before $m_{1}>0$, so that $E_{Kep}$ has non-trivial (closed) orbits of negative energy. As in Def. \ref{defi: Euclidean average} we now define

\begin{defi} The partially-averaged system is given on $\mathcal{R}$  by 
$$\tilde{E}_{Eucl, a}=E_{Kep}+\int_{S^{1}} \tilde{V}_{Eucl, 2} \, d s, $$
in which the averaging is made with respect to the closed orbits of $E_{Kep}$.
\end{defi}

\begin{prop}\label{prop: 3}
$D$ is a first integral of the partially-averaged system $\tilde{E}_{Eucl, a}$ of the Lagrange problem:
$$\{\tilde{E}_{Eucl, a}, D\}=0.$$
\end{prop}

\begin{proof} We follow the same proof as that of Prop. \ref{prop: 1}, but now proceed with single frequency averaging along closed Keplerian orbits twice. The first eliminate oscillating part of $m_{2}$-dependent terms
and the second eliminate oscillating part of $f$-dependent terms. We obtain as the proof of  Prop. \ref{prop: 1}
$$\{\tilde{E}_{Eucl, a}, D\}+O(m_{2}, f) \dot{=}0. $$
As in previous case, by Lem \ref{lem: 2.2} the contributions of order $m_{2}$ and of order $f$ from the $O(m_{2}, f)$-term has zero $\tilde{\ell}_{1}$-average, while $\{\tilde{E}_{Eucl, a}, D\}$ is independent of $\tilde{\ell}_{1}$. So we conclude as before that
$$\{\tilde{E}_{Eucl, a}, D\}=0.$$

\end{proof}

As in the two-center case, projecting the planar system to the hemisphere we get an analogue of Prop \ref{prop: 3} for Lagrange problem defined on a hemisphere. Call $\mathcal{R}_{s}$ the region such that the spherical Kepler orbits associated to $m_{1}$ are non-singular, are contained in this hemisphere and do not pass through the other Kepler center in this hemisphere.  On $\mathcal{R}_{s}$ we define the partially-averaged system $\tilde{E}_{sph, a}$ in an analogous way. This restriction is necessary as the Hooke potential becomes singular at the equator. Then with the same argument we have

\begin{prop}\label{prop: 4}
$D$ is a first integral of the partially-averaged system $\tilde{E}_{sph, a}$ on $\mathcal{R}_{sph}$ of the hemispherical Lagrange problem:
$$\{\tilde{E}_{sph, a}, D\}=0.$$
\end{prop}

%Similarly we define the partially-averaged system on the hemi-sphere and in the hyperbolic space. Since the harmonic potential is singular at the horizontal great circle, \emph{i.e.} the equator on the sphere, when the Hooke center is effective in the Lagrange problem, we do not do averaging along all spherical Keplerian orbits, but only those which does not touch the equator. Just as in Prop. \ref{prop: 3}, $D$ gives a first integral for these systems as well.

\section{The Systems in Hyperbolic Spaces}

As in Killing \cite{Killing} it is possible to define Kepler, two-center and Lagrange problems in the hyperbolic spaces in a similar fashion. The approach based on projective dynamics extends to this case as well. One way to realize the hyperbolic space is to write it as a sheet of the unit pseudosphere in $\R^{d, 1}$.  The Euclidean-hyperbolic correspondence is now again given via central projection from the origin in $\R^{d,1}$. Analogous of Kepler and Hooke potentials can be defined analogously as in the spherical case, by replacing cotangent and tangent by hyperbolic cotangent and hyperbolic tangent respectively. This way we obtain the same results also for partially averaged systems from two-center problem and Lagrange problem in a hyperbolic space: $D$ is always a first integral of the partially-averaged systems analogously defined. Here the partially-averaged system is again defined in the region such that the Keplerian orbits are closed.

 We refer to \cite{TakeuchiZhao1} for precise settings of these problems.

\section{Integrable billiards with the partially-averaged system}
 
 \subsection{Integrable Kepler billiards}

 Let $(M, g, V)$ be a natural mechanical system, in which $(M, g)$ is a Riemannian manifold and  $V: M \mapsto \R$ is the potential. Adding a codimension-1 submanifold $\mathcal{B}$ of $M$ to the data defines a system of mechanical billiards. In such a system, the particle moves under the influence of the potential $V$ and in addition reflected elastically while reaching $\mathcal{B}$. This defines a continuous, non-smooth dynamical systems. The total energy of the system is a preserved quantity of this system. As in the smooth case, the system is called integrable if it possesses $d:=\dim{M}$ preserved quantities which are independent and are in involution. In particular, if $\dim{M}=2$, then the system is integrable provided a conserved quantity in addition to the energy exists.

In the plane, now several families of mechanical billiards in a Kepler potential field have been identified. Similar results hold on the 2-dimensional sphere and in the hyperbolic plane as well. In its most general setting, we have
 
 \begin{theo}\label{thm: integrable Kepler billiards} (\cite{TakeuchiZhao1})
 In the plane, on the sphere, and in the hyperbolic plane, any finite combination of arcs from a fixed confocal family of conic sections and their degeneracies, with a Kepler center put at one of the focus, gives rise to an integrable Kepler billiard, with $D$ being the additional first integral for which $h$ is the focus-center distance.
 \end{theo}
 
Likely these are all the integrable cases with the presence of at least one effective Kepler center \cite{BlasiBarutello, BBBT}.

In higher dimensions we have
 
  \begin{theo}\label{thm: integrable Kepler billiards higher dimension} (\cite{TakeuchiZhao2})
 In the Euclidean space, consider a family of rotational invariant quadrics around a fixed axis, such that the intersections of these quadrics with any plane passing through this fixed axis belongs to the same confocal family of conic sections. The foci of the intersection is the same for all the planes through the fixed axis and are referred to as the foci of the family. On the sphere and in the hyperbolic space, consider natural analytic extensions of images of these family of quadrics from the Euclidean space under central projection. The images of the foci of the family in the Euclidean space are called the foci of the family on the sphere and in the hyperbolic space. 
 
By further putting a Kepler center at one of the common foci of this family, and consider any finite combination of open subsets from these family, we obtain integrable Kepler billiards in the Euclidean space, on the sphere and in the hyperbolic space, in which $D$ is an additional first integral.
 \end{theo}
 
 \subsection{A set of billiard systems with the partially-averaged system}
We shall generalize the notion of integrable mechanical billiard slightly. The underlying mechanical system is now replaced by a smooth, autonomous Hamiltonian system defined on $T M$ (as previously equipped with a symplectic form from $T^{*} M$ via Legendre transformation). If the configuration space $M$ has dimension 2, such a system is called \emph{integrable} if it has two independent, commuting conserved quantities.
 
 We now define a set of billiard systems in this sense:
 
In the Euclidean space, on the sphere and in the hyperbolic space, consider any finite combination of open subsets from the family of quadrics defined as in Thm. \ref{thm: integrable Kepler billiards higher dimension}. Then, any finite combination of open subsets from this family as reflection wall, with the partially-averaged system as the underlying dynamical system, defines a billiard system. This system mixes the continuous secular evolution of orbits described by the partially-regularized systems with drastic changes of orbits due to reflections with the reflection wall. 

Combining Prop \ref{prop: 1} and its analogue versions on the sphere and in the hyperbolic space with Thms \ref{thm: integrable Kepler billiards}, \ref{thm: integrable Kepler billiards higher dimension}, we obtain

\begin{theo} The Keplerian energy $E_{Kep}$ and $D$ are conserved quantities for the billiard systems defined as above. They are therefore integrable. 
\end{theo}

\begin{proof} We show that $E_{Kep}$ is a conserved quantity. First due to averaging the Hamiltonian of any partially-averaged system does not depend on the Keplerian fast angle in any local action-angle coordinates and thus the semi major axis of the Keplerian orbit is invariant. As $E_{Kep}$ is only a function of this semi major axis it is also invariant. On the other hand, $E_{Kep}$ is the sum of the kinetic and potential energies of the particle in the pure Kepler system. As both the kinetic energy and the potential energy does not change at an elastic reflection, $E_{Kep}$ is also invariant under reflections at a reflection wall. 
\end{proof}

For example, this asserts that the system in the plane, in which a particle moves according to the flow of $E_{Eucl, a}$ until it hits a straight line and gets reflected at this line has two conserved quantities, namely the Kepler energy $E_{Kep}$ and $D$. This system combines the work of \cite{Pinzari1} and \cite{GJ}. The Hamiltonian $E_{Eucl, a}$ of the system though, is not necessarily preserved under the reflections at the straight line wall. 

We have thus identified a large family of integrable billiard systems with an underlying dynamics of mechanical nature on spaces of constant curvatures. 

\medskip
\medskip
{\bf Acknowledgements:} G. P. acknowledges the ERC project 677793 (2016--2022).
L. Z. is partially supported by the DFG Heisenberg Programme ZH 605/4-1 and the Fundamental Research Funds for the Central Universities of China.

\end{document}